\documentclass{amsart}
\usepackage{amsmath,latexsym}
\usepackage[psamsfonts]{amssymb}
\usepackage{times}
\usepackage[mathcal]{euscript}
\numberwithin{equation}{section}

\renewcommand{\epsilon}{\varepsilon}

\newcommand{\be}{\begin{equation}}
\newcommand{\ee}{\end{equation}}
\newcommand{\no}{\nonumber}

\newcommand{\N}{\mathbb{N}}

\newcommand{\Q}{\mathbb{Q}}
\newcommand{\R}{\mathbb{R}}

\newcommand{\T}{\mathbb{T}}

\newcommand{\Z}{\mathbb{Z}}

\newcommand{\cE}{{\mathcal E}}
\newcommand{\cF}{{\mathcal F}}

\newcommand{\supp}{{\ensuremath{\mathrm{supp}}}}

\renewcommand{\det}{\mathop{\mathrm{det}}}

{\bf}{\it}
\newtheorem{theorem}{Theorem}[section]
\newtheorem{lemma}[theorem]{Lemma}
\newtheorem{corollary}[theorem]{Corollary}

\newtheorem{remark}[theorem]{Remark}
\newtheorem{note}[theorem]{Note}

\begin{document}
\title[The existence of bound states in a system of three
particles... ]{The existence of bound states in a system of three
particles in an optical lattice}

\author{Saidakhmat N. Lakaev,\,\,Shukhrat  S.\ Lakaev}




\begin{abstract}
We consider the hamiltonian $\mathrm{H}_{\mu},\mu\in \R$ of a system
of three-particles (two identical fermions and one different
particle) moving on the lattice ${\Z}^d ,\, d=1,2 $ interacting
through repulsive $(\mu>0)$ or attractive $(\mu<0)$ zero-range
pairwise potential $\mu V$. We prove for any $\mu\ne0$ the existence
of bound state of the discrete three-particle Schr\"odinger operator
$H_{\mu}(K),\,K\in \T^d$ being the three-particle quasi-momentum,
associated to the hamiltonian $\mathrm{H}_{\mu}$.
\end{abstract}

\maketitle
\section{Introduction}
The main goal of the paper is to prove the existence of bound states
of the three-particle Schr\"odinger operator $H_\mu(K),\,K\in \T^d$
associated to a system of three particles (two identical fermions
and one different particle) for all non-zero interactions $\mu\in
\R$.


Throughout physics, stable composite objects are usually formed by
way of attractive forces, which allow the constituents to lower
their energy by binding together. Repulsive forces separate
particles in free space. However, in structured environment such as
a periodic potential and in the absence of dissipation, stable
composite objects can exist even for repulsive interactions that
arise from the lattice band structure \cite{Nature}.

The Bose-Hubbard model which is used to describe the repulsive pairs
is the theoretical basis for applications. The work \cite{Nature}
exemplifies the important correspondence between the Bose-Hubbard
model \cite{ Bloch,Jaksch_Zoller} and atoms in optical lattices, and
helps pave the way for many more interesting developments and
applications \cite{Thalhammer}. Stable repulsively bound objects
should be viewed as a general phenomenon and their existence will be
ubiquitous in cold atoms lattice physics. They give rise also to new
potential composites with fermions \cite{Hof_Zoller} or Bose-Fermi
mixtures \cite{Lewenstein}, and can be formed in an analogous manner
with more than two particles .

Cold atoms loaded in an optical lattice provide a realization of a
quantum lattice gas. The periodicity of the potential gives rise to
a band structure for the dynamics of the atoms.

The dynamics of the ultracold atoms loaded in the lower band is well
described by the Bose-Hubbard hamiltonian.

In the continuum case due to rotational invariance the hamiltonian
separates in a free hamiltonian for the center of mass
$H_{\mathrm{free}}$ and in a hamiltonian $H_{\mathrm{rel}}$ for the
relative motion. Bound states are eigenstates of $H_{\mathrm{rel}}
$.

The fundamental difference between the  discrete and continuous
multiparticle Schr\"odin\-ger operators is that in the discrete case
the kinetic  energy operator is not rotationally invariant.

One can rather resort to a Floquet-Bloch decomposition. The
three-particle Hilbert space $ \mathcal{H} \equiv \ell ^2 ({
\Z}^d)^3 $ is represented as a direct  integral associated to the
representation of the discrete group $ {\Z} ^d $ by shift operators
\begin{equation*}
L^{2}[({\Z}^d)^3] = \int_{K\in {\T}^d} \oplus
L^{2}[({\Z}^d)^2]\eta(dK),
\end{equation*}
where $\eta(dp)=\frac {d^dp}{(2\pi)^d}$ is the (normalized) Haar
measure on the torus $\T^d$. Hence the total three-body Hamiltonian
appears to be decomposable
\begin{equation*}\label{decompose}
\mathrm{H}=\int\limits_{\T^d}\oplus H(K)\eta(dK).
\end{equation*}

The fiber hamiltonian $H(K) $ depends parametrically on the {\it
quasi-momentum} $ K \in \T^d \equiv \R^d / (2 \pi {\Z}^d .$ It is
the sum of a free part and an interaction term, both bounded and the
dependence on $K$ is continuous. Bound states $\psi_{E,K}$ are
solutions of the Schr\"odinger equation
$$H(K) \psi_{E,K} = E \psi_{E,K},  \qquad \psi_{E,K} \in \ell^2
(\Z^d)^2.$$

In this work we consider the hamiltonian $\mathrm{h}_{\mu}$ of a
system of two-particles (a fermion and different particle) and the
hamiltonian $\mathrm{H}_{\mu}$ of a system of  three-particles (two
identical fermions and one different particle) on the lattice
${\Z}^d, \, d=1,2 $ interacting through zero-range potential $\mu
V$.

We denote by ${h}_{\mu}(k),k\in\T^d$ and
${H}_{\mu}(K),K\in\T^d,d=1,2$ the  Schr\"odinger operators
corresponding to the hamiltonians $\mathrm{h}_{\mu}$ and
$\mathrm{H}_{\mu}$, respectively.

First we show the existence of a unique two-particle bound state of
${h}_{\mu}(k),k\in\T^d$ with energy lying above the top of the
essential spectrum in the case of repulsive  $(\mu>0)$ interaction
and below the bottom in the attractive ($\mu<0$) case.

Second for any non-zero interaction we establish that the essential
spectrum of the three-particle operator $H_\mu(K),\,K \in \T^d$
consists of the {\it three-particle branch}, spectrum of the
non-perturbed operator $H_0(K)$ and a non-empty {\it two-particle
branches} arising due to the eigenvalues of the two-particle
operator $h_\mu(k),\,k \in \T^d.$

Third for any interacting particles $(\mu\neq0)$  we prove the
existence of three-particle bound states with energy lying above the
top resp. below the bottom of the essential spectrum for repulsive
($\mu>0$) resp. attractive ($\mu<0$) interactions.

In addition, we derive some important properties of bound states as
well as their energies: The two-particle and three-particle bound
states $\psi_{e_\mu(k),k}$ and $\psi_{E_\mu(K),K}$ in position
representation exponentially vanishes at infinity. Moreover the
bound states $\psi_{e_\mu(k),k}$ and $\psi_{E_\mu(K),K}$ and
associated energy functions $e_\mu(k)$ and $E_\mu(K)$ are
holomorphic in $k\in\T^d$ and $K\in\T^d$ respectively.

Our third result is quite surprising, because the three identical
bosonic bound states with energy lying above the top resp. below the
bottom of the essential spectrum exist also only for repulsive
($\mu>0$) resp. attractive ($\mu<0$) interactions \cite{DLKh2015}

To our knowledge analogous results have not been published yet even
for a system of three particles interacting via attractive
potentials on Euclidean space $\R^d$.

The results for a system of two identical fermions and one different
particle theoretically predict the existence of stable attractively
and repulsively bound objects for two fermionic and one bosonic
atoms. Hopefully it can be experimentally confirmed as is done for
pair of atoms with repulsive interaction in \cite{Nature}.

Notice that these results are characteristic to the Schr\"odinger
operators associated to a system of three particles moving in one-
or two-dimensional lattice $\Z^d$ and Euclidean space $\R^d$.

We remark that the results of the paper should be hold for a system
of three particles interacting via short range attractive and
repulsive potentials.

Our results allow us to formulate the following two hypothesis on
existence of bound states for a system of three particles (two
fermions and one different particle) moving on the lattice
$\Z^d,d\geq1$ and interacting via a short-range attractive or
repulsive forces:

{\it In dimension $d=1,2$ this system has finite number of
three-particle bound states, with energies lying as above the
essential spectrum, as well as below its bottom.

In dimension $d=3$ we propose the existence of infinitely many bound
states (Efimov's effect) with energies lying as above, as well as
below the essential spectrum for special repulsive or attractive
interaction.}

This paper is organized as follows. Section 1 is introduction. In
section 2 we describe the hamiltonian of the two-body and the
three-body case in the Schr\"odinger representation. It corresponds
to the Hubbard hamiltonian in the number of particles
representation. In section 3 we introduce the Floquet-Bloch (von
Neumann) decomposition and choose relative coordinates to express
the discrete Schr\"odinger operator $H_\mu(K),\; K \in \T^d $
explicitly. In section 4 we state our main results. In section 5 we
introduce \it channel operators \rm and describe the essential
spectrum of $ H_\mu(K),\;K\in \T^d $ by means discrete spectrum of $
h_\mu(k), \; k \in \T^d $. We prove the existence of bound states in
section 6.

\section{Hamiltonian of a system of three
particles (two identical fermions and one different particle) on
lattices}
\subsection{ The coordinate representation}
 Let ${\Z}^{d},d\geq1$ be the $d$-dimensional lattice. Let
$\ell^2[({\Z}^{d})^{m}],m=2,3,...$ be Hilbert space of
square-summable functions $\,\,\hat{\varphi}$ defined on the
Cartesian power of $({\Z}^{d})^{m},d=1,2$ and
$\ell^{2,a}[(\Z^{d})^{m}]\subset \ell^{2}[(\Z^{d})^{m}]$ be the
subspace of functions antisymmetric with respect to the permutation
of the first two coordinates of the particles. Let $\Delta$ be the
lattice Laplacian, i.e., the operator which describes the transport
of a particle from one site to another site:
$$ (\Delta\hat{\psi})(x)=-\sum_{\mid s\mid=1}[
\hat{\psi}(x)-\hat{\psi}(x+s)].$$

The free Hamiltonian $\hat{\mathrm{h}}_{0,\gamma}$ of a system of
two arbitrary particles ( a fermion and different particle) on the
$d$-dimensional lattice $\Z^d,d=1,2$ acts on
$\ell^{2}[(\Z^{d})^{2}]$ and is of the
 form
 \begin{equation}\no
\hat{\mathrm{h}}_{0,\gamma}={\Delta}\otimes I +I\otimes
\gamma\Delta,
\end{equation}
where $\gamma>0$ is the ratio of the masses of fermion and different
particle.

Respectively, the free Hamiltonian $\hat{\mathrm{h}}_{0}$ of a
system of two identical fermions acts on $\ell^{2,a}[(\Z^{d})^{2}]$
and is of the form
\begin{equation}\no
\hat{\mathrm{\bf h}}_{0}={\Delta}\otimes I + I\otimes\Delta.
\end{equation}
The total Hamiltonian of a system of two arbitrary particles $\hat
{\mathrm{h}}_{\mu}$ resp. identical fermions $\hat {\mathrm{\bf
h}}_{\mu}$ with zero-range pairwise interaction $\mu\ne0$ is a
bounded perturbation of the free Hamiltonian
$\hat{\mathrm{h}}_{0,\gamma}$ resp. $\hat {\mathrm{\bf h}}_{0}$ acts
on the Hilbert space $\ell^{2}[( {\Z}^d)^2]$ resp. $\ell^{2,a}[(
{\Z}^d)^2]$ and is of the form
\begin{equation}\label{two-part}
\hat{\mathrm{h}}_{\mu}\equiv\hat{\mathrm{h}}_{\mu,\gamma}
=\hat{\mathrm{h}}_{0,\gamma}+\mu\hat v \,\, \,\mbox{resp.} \,\,
\hat{\mathrm{\bf h}}_{\mu} =\hat{\mathrm{\bf h}}_{0}+\mu\hat v
\end{equation}
where
\begin{equation*}
(\hat v\hat \psi)(x_1,x_2) = \delta _{x_1 x_2}
{\hat\psi}(x_1,x_2),\,\, {\hat\psi} \in
\ell^{2}[({\Z}^d)^2],\,\,\mbox{resp.}\,\,{\hat\psi}\in
\ell^{2,a}[({\Z}^d)^2]
\end{equation*}
and $\delta _{x_1 x_2}$ is  the Kronecker delta.

The free Hamiltonian $\widehat{\mathrm{H}}_{0,\gamma}$ of a system
of three particles (two identical fermions and one different
particle) on the $d$-dimensional lattice $\Z^d$ acts in
$\ell^{2,a}[({\Z}^d)^3]$ can be represented as
\begin{equation}\label{free0} \widehat {\mathrm{H}}_{0,\gamma}=\Delta\otimes I\otimes I
+ I \otimes \Delta \otimes I + I\otimes I\otimes \gamma\Delta.
\end{equation}
The total Hamiltonian $\widehat {\mathrm{H}}_{\mu}$ of a system of
three-particles with the pairwise zero-range interaction is a
bounded perturbation of the free Hamiltonian
$\widehat{\mathrm{H}}_{0}$
\begin{equation}\label{total}
\widehat {\mathrm{H}}_{\mu}\equiv\widehat
{\mathrm{H}}_{\mu,\gamma}=\widehat{\mathrm{H}}_{0,\gamma}+ \mu(
\widehat{V}_{12}+\widehat{V}_{13}+\widehat{V}_{23}),
\end{equation}
where $\widehat{V}_{1,2}$ and  $\widehat{V}_{\alpha,3},\,\alpha=1,2$
are multiplication operators
\begin{equation*}
(\widehat{V}_{1,2}\hat{\psi})(x_1,x_2,x_3)= \delta_{x_1
x_2}\hat{\psi}(x_1,x_2,x_3), \quad \hat{\psi} \in
\ell^{2,a}[(\Z^{d})^{3}],
\end{equation*}
and
\begin{equation*}
(\widehat{V}_{\alpha,3}\hat{\psi})(x_1,x_2,x_3)= \delta_{x_{\alpha}
x_3}\hat{\psi}(x_1,x_2,x_3), \quad \hat{\psi} \in
\ell^{2,a}[(\Z^{d})^{3}].
\end{equation*}
\begin{remark}\label{coexist}It can be easily seen that the equalities
\begin{align}
&(\hat{v}_{1,2}\hat{\psi})(x_1,x_2)=0,\,\mbox{for all}\,
\hat{\psi}\in \ell^{2,a}[(\Z^{d})^{2}],\\\nonumber
&(\widehat{V}_{1,2}\hat{\psi})(x_1,x_2,x_3)=0,\,\mbox{for all}\,
\hat{\psi}\in \ell^{2,a}[(\Z^{d})^{3}]
\end{align}
hold. It means that two interacting fermions cannot coexist on the
same site of the lattice, this is Pauli's exclusion principle for
identical fermions. Further, we deal only with the hamiltonian
$\hat{\mathrm{h}}_{\mu}$.
\end{remark}

\subsection{ The momentum  representation}
Let us rewrite our operators in the momentum representation. Let
$\T^d$ be the $d$-dimensional torus (Brillouin zone)
\begin{equation*}
\mathbb{T}^{d}=(\mathbb{R}/2\pi \mathbb{Z)}^{d}\equiv \lbrack -\pi,\pi )^{d}%
\text{ },
\end{equation*}%
the Pontryagin dual group of $\mathbb{Z}^{d}$ and $\eta(dp) =\frac
{d^dp}{(2\pi)^d}$ is the (normalized) Haar measure on the torus. Let
$L^{2,a}[({\T}^{d})^{3}]\subset L^2[({\T}^{d})^{3}]$ be the subspace
of the functions antisymmetric with respect to the permutation of
the first two coordinates of particles.

Let ${ \cF}_m:L^2[({\T}^d)^m] \rightarrow \ell^2[(
{\Z}^d)]^m),\,m\in \N$ be the standard Fourier transform and
$\cF_3^{a}$ be the restriction of $\cF_3$ on the subspace
$L^{2,a}[(\T^{d})^{3}]$.

In the momentum representation the two-and three-particle
Hamiltonian $\mathrm{h}_{\mu}$ and $\mathrm{H}_{\mu}$ is given by
the bounded self-adjoint operator
$$
 \mathrm{h}_{\mu}= \cF_2^{-1}\mathrm{\hat h}_{\mu}\cF_2
$$
and
$$
\mathrm {H}_{\mu}= [\cF_3^{a}]^{-1}\mathrm {\widehat
H}_{\mu}\cF_3^{a}$$ respectively. The operator $\mathrm {h}_{\mu}$
acts in $L^{2}[(\T^{d})^{2}]$ and is of the form
\begin{equation} \label{two} \mathrm {h}_{\mu}=\mathrm {h}_{0}+\,\mu \,v,\
\end{equation}
where
\begin{equation}(\mathrm {h}_{0}\,f)(k_{\alpha},k_{3})=[\varepsilon (k_{\alpha})+\gamma
  \varepsilon (k_{3})]\,f(k_{\alpha},k_{3}).
\end{equation}
The interaction operator $v$ acts in $L^{2}[({\T}^{d})^{2}]$ as
\begin{align*}
&(v f)(k_{\alpha},k_{3})=\int_{({\T}^{3})^{2}}\delta
(k_{\alpha}+k_{3}-k_{\alpha}^{\prime }-k_{3}^{\prime
})f(k_{\alpha}^{\prime},k_{3}^{\prime})dk_{\alpha}^{\prime }dk_{3}^{\prime}\\
&=\int_{({\T}^{3})^{2}}f(k_{\alpha}+k_{3}-k_{3}^{\prime},k_3^{\prime})dk_{3}^{\prime}
,\quad
   f\in L^{2}[({\T}^{d})^{2}],\,\alpha=1,2.
\end{align*}
The function $\varepsilon$ is of the form
\begin{equation} \label{eps form}
\varepsilon (p)=2\sum\limits_{i=1}^{d}(1-\cos {p^{(i)}}),\quad
p=(p^{(1)},...,p^{(d)})\in {\T}^{d}
\end{equation}
and  $\delta (k)$ denotes the $d-$ dimensional Dirac delta-function.

The three-particle Hamiltonian $\mathrm {H}_{\mu}$ is of the form
\begin{equation} \label{Hamilt}
\mathrm {H}_{\mu}=\mathrm {H}_{0}+\,\mu \,(\mathrm {V}_{1,3}+\mathrm
{V}_{2,3}),
\end{equation}  where  non-perturbed operator $H_{0}=H_{0,\gamma}$ acts in
$L^{2,a}((\T^d)^3)$ as
\begin{equation*} \label{free0}
(\mathrm {H}_{0}\,f)(k_{1},k_{2},k_{3})=[\varepsilon
(k_{1})+\varepsilon (k_{2})+\gamma \,\varepsilon
(k_{3})]\,f(k_{1},k_{2},k_{3}),
\end{equation*}
and the interaction operator $V_{\alpha,3},\alpha=1,2$ is given by
\begin{align*}
&(\mathrm {V}_{1,3}f)(k_{1},k_{2},k_{3})\\
&={\int\limits_{({\T}^d)^3}} \delta (k_{2}-k_{2}')\, \delta (k_{1}
+k_{3}
-k_{1}'-k_{3}')f(k_{1}',k_{2}',k_{3}')\eta(dk_{1}')\eta(dk_{2}')\eta(dk_{3}')\\
&={\int\limits_{{\T}^d}}f(k_{1}',k_{2},k_{3}
+k_{1}-k_{1}')\eta(dk_{1}'), \quad f\in L^{2,a}[({\T}^d)^3].
\end{align*}
and
\begin{align*}
&(\mathrm {V}_{2,3}f)(k_{1},k_{2},k_{3})\\
&={\int\limits_{({\T}^d)^3} } \delta (k_{1} -k_{1}')\, \delta (k_{2}
+k_{3}
-k_{2}'-k_{3}')f(k_{1}',k'_2,k'_3)\eta(dk_{1}')\eta(dk_{2}')\eta(dk_{3}')\\
&={\int\limits_{{\T}^d}}f(k_{1},k_{2}',k_{2}
+k_{3}-k_{2}')\eta(dk_{2}'), \quad f\in L^{2,a}[({\T}^d)^3].
\end{align*}

\section{Decomposition of the energy operators into von Neumann direct integrals.
 Quasi-momentum and coordinate systems}

Denote by $k=k_1+k_2\in \T^d$ and $K=k_1+k_2+k_3\in \T^d$ the {\it
two-} and {\it three-particle quasi-momenta}.  Define the sets
$\mathbb{Q}_{k}$ and $\mathbb{Q}_{K}$ as follows
\begin{align*}
&\mathbb{Q}_{k}=\{(k_1,k-k_1){\in }({\T}^d)^2: k_1 \in\T^d,
k-k_1\in\T^d\}
\end{align*}
and
\begin{align*}
&\mathbb{Q}_{K}=\{(k_1,k_2,K-k_1-k_2){\in }({\T}^d)^3: k_1,k_2
\in\T^d, K-k_1-k_2\in\T^d\}.
\end{align*}
We introduce the maps
\begin{equation*}
\pi_{2}:(\T^d)^2\to \T^d,\quad \pi_2(k_1, k_2)=k_1
\end{equation*}
and
\begin{equation*}
\pi_3:(\T^d)^3\to (\T^d)^2,\quad \pi_3(k_1, k_2, k_3)=(k_1, k_2).
\end{equation*}

Denote by  $\pi_k, k\in \T^d$ and $\pi_{K}, K\in \T^d$ the
restrictions of $\pi_2$ and $\pi_3$ onto $\Q_{k}\subset (\T^d)^2$
and $\Q_{K}\subset (\T^d)^3$ respectively, i.e.,
\begin{equation*}\label{project}\pi_{k}= \pi_2\vert_{\Q_{k}}\quad
\text{and}\quad \pi_{K}=\pi_{3}\vert_{\Q_{K}}.
\end{equation*}
\begin{remark} We note that $ \Q_{k},\,\, k \in
{\T}^d $ and
 $ \Q_{K},\,\, K \in
{\T}^d $ are $d-$ and $2d-$ dimensional manifolds isomorphic to
${\T}^d$ and ${({\T}^d)^2}$ respectively: The maps  $\pi_{k},k\in
\T^d$ and $\pi_{K}, K\in \T^d$ are bijective from
$\mathbb{Q}_{k}\subset (\T^d)^2$ and $\mathbb{Q}_{K}\subset
(\T^d)^3$ onto  $\T^d$ and $(\T^d)^2$ with
\begin{equation*}
(\pi_{k})^{-1}(k_1)=(k_1,k-k_1)
\end{equation*}
and
\begin{equation*}
(\pi_{K})^{-1}(k_1,k_2)= (k_1, k_2, K-k_1-k_2).
\end{equation*}
\end{remark} Decomposing the Hilbert spaces $ L^{2}[(\T^d)^2]$ and $
L^{2,a}[({\T}^d)^3]$  into the direct integrals
\begin{equation*}\label{tensor}
L^{2}[({\T}^d)^2] = \int_{k\in {\T}^d} \oplus L^2
(\mathbb{Q}_k)\eta(dk)
\end{equation*}
and
\begin{equation*} L^{2,a}[({\T}^d)^3] = \int_{K\in {\T}^d} \oplus
L^{2,a}[\mathbb{Q}_K]\eta(dK)\end{equation*} yield the
decompositions of the Hamiltonians $\mathrm{h}_\mu$ and
$\mathrm{H}_\mu$ into the direct integrals
\begin{equation}\label{fiber2}
\mathrm{h}_{\mu}= \int\limits_{k \in \T^d}\oplus\tilde
h_{\mu}(k)\eta(dk)
\end{equation}
and
\begin{equation}\label{fiber3} \mathrm{H}_{\mu,}=\int\limits_
{K \in {\T}^d}\oplus \tilde H_{\mu}(K)\eta(dK).
\end{equation}
\subsection{The discrete Schr\"odinger operators}

The fiber operator $\tilde h_{\mu}(k),$ $k \in {\T}^d$  from the
direct integral decomposition \eqref{fiber2} acts in $L^2
(\mathbb{Q}_k)$ and is unitarily equivalent to the operator
$h_{\mu}(k),$ $k \in {\T}^d$ given by
\begin{equation}\label{two} h_{\mu}(k)=h_{\gamma,0}(k)+\mu
v, \mu\ne0.
\end{equation}
The operator $h_{0}=h_{\gamma,0}(k)$ is the multiplication operator
by the function  $\cE_{\gamma,k}(p)$:
\begin{equation*}
(h_{\gamma,0}f)(p)=\cE_{\gamma,k}f(p),\quad f \in L^{2}(\T^d),
\end{equation*}
where
\begin{equation}\label{two-part_dispersion}
\cE_{\gamma,k}(p)=
  \varepsilon(q)+\gamma\varepsilon (k-q),
\end{equation}
and
\begin{equation*}
(vf)(p)= \int\limits_{{\T}^d}f(q)d \eta (q), \quad f \in
L^{2}({\T}^d).
\end{equation*}

The fiber operator $\tilde H_{\mu}(K),$ \,$K \in {\T}^d$ from the
direct integral decomposition \eqref{fiber3} acts in $L^{2,a}
(\mathbb{Q}_K)$ and is unitarily equivalent to the operator
$H_{\mu}(K),$ $K \in {\T}^d$ given by
\begin{equation*}\label{three-particle}
H_{\mu}(K)=H_{0}(K)+\mu (V_{13}+V_{23}).
\end{equation*}
The operator $H_{0}(K)=H_{\gamma,0}(K)$ acts in the Hilbert space
$L^{2,a}[({\T}^d)^2]$ as
\begin{equation*}\label{TotalK}
(H_{\gamma,0}(K)f)(p,q)=E(K;p,q)f(p,q),
\end{equation*}
where
\begin{equation*}E(K;p,q)=\varepsilon(p) + \varepsilon (q)+\gamma \varepsilon
(K-p-q).
\end{equation*}
The perturbation operator $\mathbb{V}=V_{13}+V_{23}$ in coordinates
$(p,q)\in (\T^d)^2$ can be written in the form

\begin{align}\label{potential}
&(\mathbb{V}f)(p,q)= \int\limits_{\T^d}f(p,t)
\eta(dt)+\int\limits_{\T^d}f(t,q)\eta(dt),f\in L^{2,a} [({\T}^d)^2]
\end{align}

\section{Statement of the main results}
According to the Weyl theorem \cite{RSIV} the essential spectrum
   $\sigma_{\mathrm{ess}}(h_{\mu}(k))$ of the
operator $h_{\mu}(k)(k),\,k \in \T^d$ coincides with the spectrum $
{\sigma}(h_{0}(k)) $ of $h_{0}(k).$ More specifically, since for any
$k\in \T^d$ the function $\cE_k(p)$ is continuous in $p\in \T^d$ the
equality
$$
\sigma_{\mathrm{ess}}(h_\mu(k))= [\cE_{\min}(k) ,\,\cE_{\max}(k)]
$$ holds, where \begin{align*} &\cE_{\min}(k)\equiv\min_{p\in
\T^d}\cE_k(p)=\varepsilon(p_{\min}(k))+\gamma\varepsilon(K-p_{\min}(k)),\,p_{\min}\in
\T^d,\\
&\cE_{\max}(k)\equiv\max_{p\in\T^d}
\cE_{k}(p)=\varepsilon(p_{\max}(k))+\gamma\varepsilon(K-p_{\max}(k)),\,p_{\max}\in
\T^d.
\end{align*}

The spectrum $\sigma(H_{0}(K))$ of the
 non-perturbed operator $H_{0}(K),\,K \in \T^d$ coincides with the
segment $[\mathrm{E}_{\min}(K),\mathrm{E}_{\max}(K)]$. Since for
each $K\in\T^d$ the function $E(K;p,q)$ is continuous and symmetric
on $(\T^d)^2,\,d=1,2$ the equalities
$$
\mathrm{E}_{\min}(K)=\min_{p,q \in \T^d}
E(K;p,q)=E(K;p_{\min}(K),q_{\min}(K))$$ and
$$\mathrm{E}_{\max}(K)=\max_{p,q \in
\T^d}E(K;p,q)=E(K;p_{\max}(K),q_{\max}(K))
$$
hold, where
$(p_{\min}(K),q_{\min}(K)),\,(p_{\max}(K),q_{\max}(K))\in (\T^d)^2$.
\begin{note}
We remark that the essential spectrum
$[\cE_{\min}(k),\cE_{\max}(k)]$ strongly depends on the
quasi-momentum $k\in\T^d;$ when $k=\vec{\pi}=(\pi,...,\pi)\in \T^d$
the essential spectrum of $h_{\mu}(k)$ degenerated to the set
consisting of a unique point $\{\cE_{\min}(\vec{\pi})=
\cE_{\max}(\vec{\pi}) =2d\}$ and hence the essential spectrum of
$h_{\mu}(k)$ is not absolutely continuous for all $k\in \T^d.$
Similar arguments should be true for the spectrum of $H_0(K)$.
\end{note}

The following theorem asserts the existence of a unique eigenvalue
$e_{\mu}(k)$ of the operator $h_{\mu}(k)$, which lays above the top
$\cE_{\max}(k)$ resp. below the bottom $\cE_{\min}(k)$ of the
essential spectrum $\sigma_{\mathrm{ess}}(h_\mu (k))$ for repulsive
($\mu>0$) resp. attractive ($\mu<0$) interactions.

\begin{theorem}\label{existencetwo}Let $d=1,2$. For any $\mu\ne0$ the operator $h_{\mu}(k),k\in
\T^d$ has a unique eigenvalue $e_{\mu}(k)$, which satisfies the
relations:
$$e_{\mu}(k)>\cE_{\max}(k),\,k\in\T^d\,\, \mbox{and}\,\,
e_{\mu}(0)>e_{\mu}(k), k\in \T^d\setminus\{0\} \, \,\mbox{for}\,\,
\mu>0
$$
and
$$e_{\mu}(k)<\cE_{\min}(k),\,k\in\T^d \,\, \mbox{and}\,\, e_{\mu}(0)<e_{\mu}(k),
 k\in\T^d\setminus\{0\} \, \,\mbox{for}\,\,
\mu<0.
$$ The eigenvalue $e_{\mu}(k)$ is
holomorphic function in $k\in\T^d$ and for any $k\in\T^d$ the
associated eigenfunction $f_{\mu,e_{\mu}(k)}(p)$ is holomorphic in
$p\in{\T}^d$ and is of the form
\begin{equation*}\label{eigen}
f_{\mu,e_{\mu}(k)}(\cdot)=\frac{\mu
c(k)}{e_{\mu}(k)-\cE_{k}(\cdot)}\,\,\mbox{resp.}\,\,f_{\mu,e_{\mu}(k)}(\cdot)=\frac{\mu
c(k)}{\cE_{k}(\cdot)-e_{\mu}(k)},
\end{equation*}
where $c(k)\neq 0$ is  a normalizing constant. Moreover, the vector
valued mapping
\begin{equation*}\label{map}
f_{\mu}:\mathbb{\T}^d \rightarrow
L^2[\mathbb{\T}^d,\eta(dk);L^{2}({\T}^d)],\,k\rightarrow
f_{\mu,e_{\mu}(k)}
\end{equation*}
is holomorphic on $\mathbb{\T}^d$.
\end{theorem}
Theorem \ref{existencetwo} can be proven in the same way as Theorem
\ref{existencetwo} in \cite{LU12}.

The essential spectrum of the three-particle operator $H_\mu(K),\,K
\in \T^d$ is described by the spectrum of the non perturbed operator
$H_0(K)$ and the discrete spectrum of the two-particle operator
$h_\mu(k),\,k \in \T^d.$

\begin{theorem}\label{ess} Let $d=1,2$. For
any $\mu\ne0$ the essential spectrum $\sigma_{\mathrm{ess}}(H_\mu
(K))$ of $H_\mu(K)$ satisfies the following relations
$$
\sigma_{\mathrm{ess}}(H_\mu (K))=\cup _{k\in
{\T}^d}\{e_\mu(K-k)+\varepsilon(k)\} \cup
[E_{\min}(K),E_{\max}(K)]\supset [E_{\min}(K),E_{\max}(K)],
$$
where $e_\mu(k)$ is the unique eigenvalue of the operator
$h_\mu(k),k\in \T^d$.
\end{theorem}
Theorem \ref{ess} can be proven in the same way as Theorem \ref{ess}
in \cite{DLKh2015}.

Let $\tau^{t}_{\mathrm{ess}}(H_\mu (K))$ resp.
$\tau^{b}_{\mathrm{ess}}(H_\mu (K))$ be the top  resp. the bottom of
the essential spectrum $\sigma_{\mathrm{ess}}(H_\mu (K))$.

Our main theorem asserts that the operator $H_{\mu}(k),K\in \T^d$
has eigenvalue for all repulsive ($\mu>0$) and attractive ($\mu<0$)
forces.

\begin{theorem}\label{existencethree} Let $d=1,2$.
For all $\mu\ne0$ and $K \in \T^d$  the operator $H_\mu(K)$ has
eigenvalue lying outside of the essential spectrum
$\sigma_{\mathrm{ess}}(H_\mu (K))$. Moreover, the eigenvalue
$E_\mu(K)$ is lying above the top $\tau^{t}_{\mathrm{ess}}(H_\mu
(K))$ for repulsive ($\mu>0$) interaction and below the bottom
$\tau^{b}_{\mathrm{ess}}(H_\mu (K))$ of $\sigma_{\mathrm{ess}}(H_\mu
(K))$ for attractive ($\mu<0$).


Any eigenvalue $E_\mu(K)$ of $H_\mu(K)$ is a holomorphic function in
$K\in \T^d$. The associated eigenfunction (bound state)
$f_{\mu,E_\mu(K)}(\cdot,\cdot)\in L^{2,a}[({\T}^d)^2]$ is
holomorphic in $(p,q)\in({\T}^d)^2$ and the vector valued mapping
\begin{equation*}\label{map}
f_{\mu}:\mathbb{\T}^d \rightarrow
L^2[\mathbb{\T}^d,\eta(dK);L^{2,a}[({\T}^d)^2]],\,K\rightarrow
f_{\mu,E_\mu(K)}
\end{equation*} is also holomorphic in $K\in\mathbb{\T}^d$.
\end{theorem}

Theorem \ref{existencethree} yields the following corollary, which
asserts the existence of a band spectrum for two and three
interacting particles on the lattice $\Z^d, d=1,2$.
\begin{corollary}\label{existenceband} For any $\mu\ne0$ the
two- resp. three-particle hamiltonian $\mathrm{h_\mu}$
resp.$\mathrm{H_\mu}$ has a band spectrum
\begin{equation*}
[\min_{k\in\T^d} e_\mu(k),\max_{k\in\T^d}
e_\mu(k)]\,\,\mbox{resp.}\,\,[\min_{K\in\T^d}E_\mu(K),\max_{K\in\T^d}E_\mu(K)].
\end{equation*}
\end{corollary}
\begin{note}\label{Rem_exist}
For any $\mu\ne0$ Theorems \ref{existencetwo} and \ref{ess} yield
that  the two-particle essential spectrum
$$
\sigma_{\mathrm{esstwo}}(H_\mu (K))=\cup _{k\in
{\T}^d}\{e_\mu(K-k)+\varepsilon (k)\}$$ of the operator $H_{\mu}(K),
K\in {\T}^d$ is a non empty set and hence
\begin{equation*}\label{d=1or d= 2}
\tau^{b}_{\mathrm{ess}}(H_{\mu}(K))<E_{\min}(K)\, \,\mbox{for}\,\,
\mu<0
\end{equation*}
and
\begin{equation*}\label{d=1or d= 2}
\tau^{t}_{\mathrm{ess}}(H_{\mu}(K))>E_{\max}(K),\, \,\mbox{for},\,\,
\mu>0
\end{equation*}
which allows the existence of bound states of three (two identical
fermions and one  different particle) repulsively resp. attractively
interacting particles on the lattice $\Z^d$ \cite{ALzM04,LSN93}.

We note that this result is characteristic to the Schr\"odinger
operators associated to a system of three particles moving in a one-
or two-dimensional space.
\end{note}

\section{The essential spectrum of the operator $ H_\mu(K)$}
Since the particles are identical there is only one channel operator
 $ H_{\mu,ch} (K), K {\ \in } \T^d,\,d=1,2$ defined in the
Hilbert space $L^{2}[({\T}^d )^2]$ as
\begin{equation*}
 H_{\mu,ch}(K)=H_0(K)+\mu V.
\end{equation*}

The operators $H_0(K)$ and $V=V_{1,3}=-V_{2,3}$ act as follows
\begin{equation*}\label{TotalK}
(H_0(K)f)(p,q)=E(K;p,q)f(p,q),\quad f\in L^{2}[({\T}^d )^2],
\end{equation*}
where
\begin{equation*}\label{Eps}
E(K;p,q)=\varepsilon  (p) + \varepsilon (q)+\gamma \varepsilon
(K-p-q)
\end{equation*}
and
 \begin{equation*}\label{Poten}(V
f)(p,q)= \int\limits_{\T^d}f(p,t)\eta(dt),\quad f\in
L^{2}[({\T}^d)^2].
\end{equation*}

The decomposition of the space $L^{2}[(\T^d)^2]$ into the direct
integral

 $$L^{2}[({\T}^d )^2]= \int\limits_{k\in \T^d}
\oplus L^{2}( \T^d) \eta(dp)
$$
yields the decomposition
 $$H_{\mu,ch}(K)=\int\limits_{k\in \T^d}
 \oplus h_{\mu}(K,p) \eta(dp).$$

The fiber operator $h_{\mu}(K,p)$ acts in the Hilbert space
$L^{2}(\T^d)$ and is of the form
\begin{equation}\label{representation}
h_{\mu}(K,p)={h}_{\mu} (K-p)+\varepsilon(p) I,
\end{equation} where $I=I_{L^{2}(\T^d)}$ is
the identity operator and the operator ${h}_{\mu}(K-p)$ is defined
by \eqref{two}.

The representation \eqref{representation} of the operator
$h_{\mu}(K,p)$ and Theorem \ref{existencetwo} yield the following
description for the spectrum of $h_{\mu}(K,p)$
\begin{align*}\label{stucture}
 &\sigma (h_{\mu}(K,p))
 =[e_\mu(K-p)+\varepsilon(p)]\cup \big[E_{\text{min}}(K,p),E_{\text{max}}(K,p)
 \big],
\end{align*}
where $E_{\min}(K,p)=\min_{q\in\T^d}E(K,p\,;q)$ and
$E_{\max}(K,p)=\max_{q\in\T^d}E(K,p\,;q)$.

Notice that Theorem \ref{existencetwo} yields the result, which
states that for any $\mu\ne0$ the essential spectrum
$\sigma_{\mathrm{ess}}(H_{\mu}(K))$ of the operator $H_{\mu}(K)$ is
different from the spectrum of the non-perturbed operator $H_0(K)$.

\begin{lemma}\label{inequality} For any  $K\in \T^d$
the following inequalities hold
$$\tau^{b}_{\mathrm{ess}}(H_{\mu}(K))<
\tau_{\mathrm{essthree}}(H_{\mu}(K))=E_{\min}(K) \,\, \mbox{for}\,
\, \mu<0$$ and
$$\tau^{t}_{\mathrm{ess}}(H_{\mu}(K))>
\tau_{\mathrm{essthree}}(H_{\mu}(K))=E_{\max}(K)\,\, \mbox{for}\,\,
\mu>0.$$
\end{lemma}
\begin{proof} Theorem \ref{existencetwo} yields that for any
$\mu<0$ and $k\in\T^d$ the operator $h_\mu(k)$ has a unique
eigenvalue $e_\mu(k)<\cE_{\min}(k).$ Set
$$
 Z_\mu (K,p) =e_{\mu}(K-p)+\epsilon(p).
$$
The definition of $\tau_{\mathrm{ess}}(H_{\mu}(K))$ gives
\begin{align*}
&\tau_{\mathrm{ess}}(H_{\mu}(K))=\inf_{p\in\T^d}Z_\mu(K,p)\\
&\leq \inf_{p\in\T^d}[e_{\mu}(K-p_{\min}(K))+
\varepsilon(p_{\min}(K))]<
\cE_{\min}(K-p_{\min}(K))+\varepsilon(p_{\min}(K))=E_{\min}(K)
\end{align*} which proves
Lemma \ref{inequality}.
\end{proof}
\section{Proof of the main results}
Let
\begin{align*}
&E_{\min}(K,k)=\min_{q\in\T^d}E(K,k\,;q)=\min_{q\in\T^d}\cE_k(q)+\varepsilon(K-k),\\
&E_{\max}(K,k)=\max_{q\in\T^d}E(K,k\,;q)=\max_{q\in\T^d}\cE_k(q)+\varepsilon(K-k).
\end{align*}
For any $\mu\in \R$ and $K,k\in \T^d,d=1,2$ the determinant
$\Delta_\mu (K,k\,;z)$ associated to the operator $h_{\mu}(K,k)$ can
be defined as a real-analytic function in $\mathrm{C}\setminus
[E_{\min }(K,k),\,E_{\max }(K,k)]$ by
\begin{equation*}\label{determinant}
\Delta_\mu (K,k\,; z ) = 1+\mu \int\limits_{\T^d}\frac{\eta(dq)}{E
(K; k\,,q)-z}.
\end{equation*}
\begin{lemma}\label{nollar2}
For any $\mu\in \R$ and  $K,k\in\T^d$ the number  $z \in
{\mathrm{C}} {\setminus} [E_{\min }(K,k),\,E_{\max}(K,k)]$ is an
eigenvalue of the operator $h_{\mu}(K,k) $ if and only if
$$
 \Delta_\mu (K, k\,; z) = 0.
$$
 \end{lemma}
The proof of Lemma \ref{nollar2}  is simple and can be proven in the
same way as Lemma in \cite{ALKh2012}.
\begin{remark}
We note that for each $\mu\ne0$ and $K,k \in \T^d$ there exist
either $z_{l}=z_{l}(K,k)<E_{\min }(K,k)$ or
$z_{r}=z_{r}(K,k)>E_{\max}(K,k)$ such that either for all $$z\in
(-\infty, z_{l})\cup[E_{\max}(K,k),+\infty)$$ or $$z\in
(-\infty,E_{\max}(K,k)]\cup[z_{r},+\infty)$$ the function
$\Delta_\mu (K,k\,;z)$ is non-negative and the square root function
$\Delta^{\frac{1}{2}}_\mu (K,k\,;z)$ is well defined.
\end{remark}
We define for each $\mu \in\R$ and $z\in \R\setminus
\big[\tau^b_{\mathrm{ess}}(H_{\mu}(K)),\tau^t_{\mathrm{ess}}(H_{\mu}(K))\big]$
the self-adjoint compact Birman-Schwinger operator $\mathrm{L}_\mu(
K,z),\,K\in \T^d$ as
\begin{equation}\label{compact_operator}
[\mathrm{L}_\mu(K,z)\psi](p)=\mu
 \int\limits_{\T^d} \frac{\Delta^{-\frac{1}{2}}_\mu(K,p,z)
\Delta^{-\frac{1}{2}}_\mu(K,q, z)}{E(K;p,q)-z}\psi(q)
\eta(dq),\\
\psi\in L^2(\T^d).
\end{equation}

Notice that for $\mu<0$ the operator $\mathrm{L}_\mu(K,z),\, z <
\tau^b_{\mathrm{ess}}(H_{\mu}(K))$ has been introduced in
\cite{LSN93} to investigate Efimov's effect for the three-particle
lattice Schr\"{o}dinger operator $H_\mu(K)$ associated to a system
of two identical fermions and one different particle on the lattice
$\Z^3$.

\begin{lemma}\label{eigenvalue}
Let $\mu>0$ and $z>\tau^t_{\mathrm{ess}}(H_{\mu}(K))$ resp. $\mu<0$
and $z<\tau^b_{\mathrm{ess}}(H_{\mu}(K))$.
 The following assertions ~(i)--(ii) hold true.
\item[(i)]
If $f\in L^{2,a}[({\T}^d)^2]$ solves the equation $H_\mu(K)f = zf$,
then
$$\psi(p)=\Delta^{\frac12}_\mu(K,p\:;z)\int\limits_{\T^d}f(p,t)\eta(dq)\in L^2({\T}^d)$$ solves
$\mathrm{L_\mu}(k,z)\psi=\psi$.
\item[(ii)]
If $\psi \in L^2({\T}^d)$ solves $ \mathrm{L_\mu}(K,z)\psi=\psi$,
then
\begin{equation}\label{solution}
f(p,q)=-\dfrac{\mu[\varphi(p)-\varphi(q)]}{E(K;p,q)-z}\in
L^{2,a}[({\T}^d)^2]
\end{equation}
solves the equation $H_\mu(K)f = zf$, where
$\varphi(p)=\Delta^{-\frac12}_\mu(K,p\:;z)\psi(p)$.
\end{lemma}
\begin{proof}
\item[(i)] Let $\mu>0$. Assume that for some
$z>\tau^{t}_{\mathrm{ess}}(H_{\mu}(K)), K\in\T^d$ the equation
$$(H_{\mu}(K)f)(p,q)=z f(p,q),$$ i.e., the equation
\begin{equation}
[E(K;p,q)-z]f(p,q)=-\mu\int\limits_{\T^d}f(p,t)\eta(dt)+\mu\int\limits_{\T^d}f(q,t)\eta(dt)\nonumber
\end{equation}
has a solution $f\in L^{2,a}[({\T}^d)^2]$. Write
$$\varphi(p)=\int\limits_{\T^d}f(p,q)\eta(dq) \in L^2({\T}^d).$$
Then we have the following representation
\begin{equation}\label{solution}
f(p,q)=-\dfrac{\mu[\varphi(p)-\varphi(q)]}{E(K;p,q)-z}\in
L^{2,a}[({\T}^d)^2],
\end{equation}
which gives the equation
\begin{equation}\label{Phi}
\varphi(p)[1+\mu\int\limits_{\T^d}\frac{\eta(dp)}{E(K;p,q)-z}]=
\mu\int\limits_{\T^d}\frac{\varphi(q)\eta(dq)}{E(K;p,q)-z} .
\end{equation}
Taking into account $\Delta_\mu(K,p\:;z)>0$ for
$z>\tau^{t}_{\mathrm{ess}}(H_{\mu}(K)), K\in\T^d$ and denoting by
$\psi(q)=\Delta^{\frac12}_\mu(K,q\:;z)\varphi(q)$ we get the
equation
\begin{equation}\label{B-Sequation}
\mu \int\limits_{\T^d} \frac{\Delta^{-\frac{1}{2}}_\mu(K,q,
z)\Delta^{-\frac{1}{2}}_\mu(K,p,z) }{E(K;p,q)-z}\psi(p)
\eta(dp)=\psi(q),
\end{equation} i.e.,
$\mathrm{L_\mu}(k,z)\psi=\psi$.
\item[(ii)] Assume that $\psi$ is a solution of equation
\eqref{B-Sequation}.Then the function
\begin{equation}\label{function}
\varphi(p)=\Delta^{-\frac12}_\mu (K,p\,;z)\psi(p)
\end{equation}
is a solution of equation \eqref{Phi} and hence the function defined
by \eqref{solution} is a solution of the equation $H_\mu(K)f = zf$,
i.e., is an eigenfunction of the operator $H_{\mu}(K)$ associated to
the eigenvalue $z>\tau^{t}_{\mathrm{ess}}(H_{\mu}(K)).$

The case $\mu<0$ and $z<\tau^{b}_{\mathrm{ess}}(H_{\mu}(K))$ of
Lemma \ref{eigenvalue} can be proven in the same way.
\end{proof}

{\bf Proof of Theorem \ref{ess}.} The theorem  can be proven,
applying Lemma \ref{inequality}, in the same way as Theorem 3.2 in
\cite{ALzM04}.

{\bf Proof of Theorem \ref{existencethree}.} Let $\mu>0$ and $K\in
\T^d, d=1,2$. For any non-negative $f \in L^2(\T^d)$ and
$z>\tau^{t}_{\mathrm{ess}}(H_{\mu}(K))$ the following relations
\begin{align}\label{norm}
&(\mathrm{L_\mu}(K,z)f,f)=-\mu\int\limits_{\T^d}
\int\limits_{\T^d}\frac{f(p)\overline{f(q)}\eta(dp)\eta(dq)}
{\Delta^{\frac{1}{2}}_\mu(K,p,z)\Delta^{\frac{1}{2}}_\mu(K,q,z)(z-E(K;p,q))}\\
&<\frac{-\mu}{z-E_{\min}}\int\limits_{\T^d}
\int\limits_{\T^d}\frac{f(p)\overline{f(q)}\eta(dp)\eta(dq)}
{\Delta^{\frac{1}{2}}_\mu(K,p,z)
\Delta^{\frac{1}{2}}_\mu(K,q,z)}=\frac{-\mu}{z-E_{\min}} \Big
|\int\limits_{\T^d} \frac{f(p)\eta(dp)}
{\Delta^{\frac{1}{2}}_\mu(K,p,z)}\Big |^2\leq0 \nonumber
\end{align}
hold, where
$$\Delta_\mu(K,p,z)=1-\mu\int\limits_{\T^d}\frac{\eta(dq)}
{z-E(K;p,q)}.$$ Let
\begin{equation*}F_z(f)=\int\limits_{\T^d} \frac{ f(p)\eta(dp) }
{\Delta^{\frac{1}{2}}_\mu(K,p,z)},\,z>\tau^{t}_{\mathrm{ess}}(H_{\mu}(K))
\end{equation*}
be  linear bounded functional defined on $L^2(\T^d)$. According the
Riesz theorem
$$||F_z||=[\int\limits_{\T^d} \frac{\eta(dp)}{\Delta_\mu(K,p,z)}]^{\frac{1}{2}}=||{\psi}_z||
.$$ Let $\mathfrak{M}_+\subset L^2(\T^d)$ be subset of all
non-negative functions. Then
\begin{equation}
||F_z||_{L^2(\T^d)}\geq||F_z||_{\mathfrak{M}_+}\geq||{\psi}_z||.
\end{equation}
Since for any $p\in \T^d$ the function $\Delta_\mu(K,p,z)$ is
monotone decreasing in
$z\in(\tau^{t}_{\mathrm{ess}}(H_{\mu}(K)),+\infty)$ there exists
a.e. the point-wise limit
$$\lim_{z\to
\tau^{t}_{\mathrm{ess}}(H_{\mu}(K))}\frac{1}{\Delta_\mu(K,p\,;z)}
=\frac{1}{\Delta_\mu(K,p\,;\tau^{t}_{\mathrm{ess}}(H_{\mu}(K)))}.$$
The Fatou's theorem yields the inequality
$$\int\limits_{\T^d} \frac{\eta(dp)
}{\Delta_\mu(K,p\,;\tau^{t}_{\mathrm{ess}}(H_{\mu}(K)))}\leq
\liminf_{z\to\tau^{t}_{\mathrm{ess}} (H_{\mu}(K))}\int\limits_{\T^d}
\frac{\eta(dp) }{\Delta_\mu(K,p\,;z)}.$$

Let $p_\mu(K)\in \T^d,\,K \in \T^d$ be a minimum point of the
function $Z_\mu(K,p),\,K \in \T^d$ on $\T^d$. Then $Z_\mu(K,p)$ has
the following asymptotics
\begin{equation}\label{Z}
Z_\mu(K,p)=\tau^{t}_{\mathrm{ess}}(H_{\mu}(K))-
(B(K)(p-p_\mu(K)),p-p_\mu(K)) +o(|p-p_\mu(K)|^2),
\end{equation}
as $|p-p_\mu(K)| \to 0,$ where $\tau^{t}_{\mathrm{ess}}
(H_{\mu}(K))=Z_\mu(K,p_\mu(K)).$

For any $K,p \in \T^d$ there exists a $\gamma=\gamma(K,p)>0$
neighborhood $W_\gamma(Z_\mu(K,p))$ of the point $Z_\mu(K,p)\in
\mathrm{C}$ such that for all $z \in W_\gamma(Z_\mu(K,p))$ the
following equality holds
\begin{equation*}
\Delta_\mu (K,p,z)=\sum_{n=1}^{\infty}C_n(\mu,
K,p)[z-Z_\mu(K,p)]^n,\\
\end{equation*}
where
\begin{align*}
C_1(\mu,K,p)=\mu \int\limits_{\T^d}
\dfrac{\eta(dq)}{[Z_\mu(K,p)-E(K;p,q)]^2}>0.
\end{align*}
From here one concludes that for any $K\in U_\delta(0)$ there is
$U_{\delta(K)}(p_{\mu}(K))$ so that for all $p\in
U_{\delta(K)}(p_{\mu}(K))$ the equality
\begin{equation}\label{nondeger}
\Delta_\mu(K,p,\tau^{b}_{\mathrm{ess}} (H_{\mu}(K))_\mu(K))
=(Z_\mu(K,p)-\tau^{t}_{\mathrm{ess}} (H_{\mu}(K))\hat
\Delta_\mu(K,p,\tau^{t}_{\mathrm{ess}} (H_{\mu}(K))
 \end{equation}
holds. Putting \eqref{Z} into \eqref{nondeger} yields the estimate
\begin{equation*}\label{otsenka2}
\Delta_\mu (K,p,\tau^{t}_{\mathrm{ess}} (H_{\mu}(K))\leq
M(K)|p-p_\mu(K)|^2.
\end{equation*}
Hence, we have
$$\int\limits_{\T^d} \frac{\eta(dp)
}{\Delta_\mu(K,p\,;\tau^{t}_{\mathrm{ess}} (H_{\mu}(K))}=+\infty.$$
Consequently, for any $P>0$ there exists $z_0>
\tau^{t}_{\mathrm{ess}} (H_{\mu}(K)) $ such that the inequality
\begin{equation}\label{norm}
||F_{z_0}||=\supp_{||\psi||=1}(F_{z_0}\psi,\psi)=[\int\limits_{\T^d}
\frac{\eta(dp)} {\Delta_\mu(K,p,z)}]^{\frac{1}{2}}>P
\end{equation} holds.
Since for all $z\geq \tau^{t}_{\mathrm{ess}}(H_{\mu}(K))$ the
positive function $(z-E_{\min})^{-1}$ is bounded above, the
inequality \eqref{norm} yields  the existence of $\psi \in
L^2(\T^d),\,||\psi||_{L^2(\T^d)}=1$ satisfying the inequality
$(\mathrm{L}_{\mu}(K,z_0)\psi,\psi)<-1$. At the same time
$$(\mathrm{L}_{\mu}(K,z)\psi,\psi)\rightarrow 0\,\, \mbox{as}\,\, z\rightarrow
+\infty.$$ Therefor there exists
$E_\mu(K)>z_0>\tau^{t}_{\mathrm{ess}}(H_{\mu}(K))$, such that
$$|(\mathrm{L}_{\mu}(K,E_\mu(K))\psi,\psi)|=1$$ and
hence the Hilbert-Schmidt theorem yields that the equation
\begin{equation}\label{ regular_solution}
\mathrm{L}_{\mu}(K,E_\mu(K))\psi=\psi
\end{equation}
has a solution $\psi \in L^2(\T^d),\,||\psi||=1$. Lemma
\ref{eigenvalue} yields that
$E_\mu(K)>\tau^{t}_{\mathrm{ess}}(H_{\mu}(K))$ is an eigenvalue of
the operator $H_{\mu}(K)$ and the associated eigenfunction
$f_{E_\mu(K)}(K;p,q)$ takes the form
\begin{equation}\label{eigenfunction}
f_{E_\mu(K)}(K;p,q)=\dfrac{\mu
c(K)[\varphi(p)-\varphi(q)]}{E_\mu(K)-E(K;p,q)}\in
L^{2,a}[({\T}^d)^2]
\end{equation}
with $c(K)= ||f_{E_\mu(K)}(K;p,q)||^{-1},\,K \in \T^d$ being a
normalizing constant.

Since for any $K\in\T^d$ the functions $\Delta_\mu(K,p\,;E_\mu(K))$
and $E_\mu(K)-E(K;p,q)>0$ are regular in $p,q \in \T^d$ the solution
$\psi$ of the equation \eqref{B-Sequation} and the function
$\varphi$ given in \eqref{Phi} are regular in $p\in \T^d$. Hence,
the eigenfunction \eqref{eigenfunction} of the operator $H_{\mu}(K)$
associated to the eigenvalue
$E_\mu(K)>\tau^{t}_{\mathrm{ess}}(H_{\mu}(K))$ is also regular in
$p,q\in \T^d$.

For any $z>\tau^{t}_{\mathrm{ess}}(H_{\mu}(K))$ the kernel function
\begin{equation}\label{kernel}
\mathrm{L}_{\mu}(K,z;p,q)=-\mu
\frac{\Delta^{-\frac{1}{2}}_\mu(K,p,z)
\Delta^{-\frac{1}{2}}_\mu(K,q, z)}{z-E(K;p,q)}
\end{equation}
of the compact self-adjoint operator $\mathrm{L}_{\mu}(K,z)$ is
regular in $p,q \in \T^d$. The Fredholm determinant
$D_\mu(K,z)=\det[I-\mathrm{L}_{\mu}(K,z)]$ associated to the kernel
function \ref{kernel} is real and regular in $z\in
(\tau^{t}_{\mathrm{ess}}(H_{\mu}(K)),+\infty)$. Lemma
\ref{eigenvalue} and the Fredholm theorem yield that each eigenvalue
of the operator $H_{\mu}(K)$ is a zero of the determinant
$D_\mu(K,z)$ and vice versa. Consequently, the compactness of the
torus $\T^d$ and the implicit function theorem give that the
eigenvalue $E_\mu(K)$ of $H_{\mu}(K), \mu>0$ is a regular function
in $K \in \T^d,\,d=1,2$.

Since for any $p,q\in\T^d$ the functions
$\Delta_\mu(K,p\:;E_\mu(K))$ and $E(K;p,q)-E_\mu(K)$ are regular in
$K\in \T^d$ the solution $\psi$ of \eqref{B-Sequation} and the
function $\varphi$ defined by \eqref{Phi} are regular in $K\in
\T^d$. Hence, the eigenfunction \eqref{eigenfunction} of the
operator $H_{\mu}(K)$ associated to the eigenvalue
$E_\mu(K)>\tau^{t}_{\mathrm{ess}}(H_{\mu}(K))$ is also regular in
$K\in \T^d$. Consequently, the vector valued mapping
\begin{equation*}\label{map}
f_{\mu}:\mathbb{T}^d \rightarrow
L^2[\mathbb{T}^d,\eta(dK);L^{2,a}[({\T}^d)^2]],\,K\rightarrow
f_{\mu,K}(\cdot,\cdot)
\end{equation*} is regular in $\mathbb{T}^d$.

Now we prove that the operator $H_{\mu}(K)$ has no eigenvalue lying
below the bottom of the essential spectrum for $\mu>0$ and $K\in
\T^d, d=1,2$.

The operator $H_{\mu}(K)$ acting in the Hilbert space
$L^{2,a}[({\T}^d)^2]$ is of the form
\begin{equation*}\label{TotalK}
(H_{\mu}(K)f)(p,q)=E(K;p,q)f(p,q)+\mu[\int\limits_{\T^d}f(p,t)
\eta(dt)+\int\limits_{\T^d}f(t,q)\eta(dt)]
\end{equation*}
where
\begin{equation*}E(K;p,q)=\varepsilon(p) + \varepsilon (q)+\gamma \varepsilon
(K-p-q).
\end{equation*}
Then
\begin{align*}\label{TotalK}
&(H_{\mu}(K)f,f)=\int\limits_{\T^d}E(K;p,q)|f(p,q)|^2\eta(dq)\eta(dp)\\
&+\mu\int\limits_{(\T^d)^2}[\int\limits_{\T^d}f(p,t)
\eta(dt)+\int\limits_{\T^d}f(t,q)\eta(dt)]\overline{f(p,q)}\eta(dq)\eta(dp).
\end{align*}
Fubini's theorem gives us the following relations
\begin{align*}
&(H_{\mu}(K)f,f)\\
&=\int\limits_{(\T^d)^2}E(K;p,q)|f(p,q)|^2\eta(dp)\eta(dq)\\
&+\mu\int\limits_{(\T^d)^2}[\int\limits_{\T^d}f(p,t)
\eta(dt)]\overline{f(p,q)}\eta(dq)\eta(dp)\\
&+\mu\int\limits_{(\T^d)^2}[\int\limits_{\T^d}f(t,q)
\eta(dt)]\overline{f(p,q)}\eta(dp)\eta(dq)\\
&=\int\limits_{(\T^d)^2}E(K;p,q)|f(p,q)|^2\eta(dp)\eta(dq)\\
&+\mu\int\limits_{\T^d}|\int\limits_{\T^d}f(p,t)\eta(dt)|^2\eta(dp)\\
&+\mu\int\limits_{\T^d}|\int\limits_{\T^d}f(t,q)\eta(dt)|^2\eta(dq)\\
&\geq\int\limits_{(\T^d)^2}E(K;p,q)|f(p,q)|^2\eta(dp)\eta(dq)\geq0.\\
\end{align*}
The min-max principle complites the proof.

Note that the case $\mu<0$ of Theorem \ref{existencethree} can be
proven in the same way as above\cite{DLKh2015}. $\square$

 {\bf Acknowledgement} This work was supported by the Grant
F4-FA-F079 of Fundamental Science Foundation of Uzbekistan.

\end{document}